\documentclass[12pt]{article}
\usepackage{amsmath, amsfonts, amssymb, latexsym}
\usepackage{amsthm}

\newcommand{\C}{\mathcal C}
\newcommand{\F}{\mathcal F}
\newcommand{\df}{d.f$.$}

\renewcommand{\d}{\textnormal{d}}

\numberwithin{equation}{section}

\begin{document}

\newtheorem{thm}{Theorem}[section]
\newtheorem{lem}{Lemma}[section]
\newtheorem{cor}{Corollary}[section]
\newtheorem{prop}{Proposition}[section]

\newtheorem{defn}{Definition}[section]
\newtheorem{exmp}{Example}[section]
\newtheorem{rem}{Remark}[section]

\title{Copulas: compatibility and Fr\'echet classes}
\author{Fabrizio Durante$^{1}$,
Erich Peter Klement$^{1}$, \\
Jos\'e Juan Quesada-Molina$^{2}$}
\date{}

\maketitle

\begin{center}
$^{1}$ Department of Knowledge-Based Mathematical Systems \\
Johannes Kepler University, A-4040 Linz, Austria \\
e-mails: fabrizio.durante@jku.at, ep.klement@jku.at\\
\end{center}

\begin{center}
$^{2}$ Departamento de Matem\'atica Aplicada\\
Universidad de Granada, E-18071 Granada, Spain\\
e-mail: jquesada@ugr.es\end{center}

\medskip

\begin{abstract}
\noindent We determine under which conditions three bivariate copulas $C_{12}$, $C_{13}$ and $C_{23}$ are compatible, 
viz. they are the bivariate marginals of the same trivariate copula $\widetilde C$, and, then, construct the class of 
these copulas. In particular, the upper and lower bounds for this class of trivariate copulas are determined. 


\end{abstract}
 
\section{Introduction}\label{S:intro}

For many years, a problem of interest to statisticians has been the construction of multivariate distribution 
functions (briefly, \df's) with given univariate marginals and some useful properties such as a simple analytic
expression and a statistical interpretation. 

One of the possible extensions of this problem is to construct 
$n$--dimensio\-nal \df's with $k$ given $m$--dimensional marginals, $1\le m<n$ and $1\le k\le \binom{n}{m}$. 
For example, given two bivariate \df's $F_{12}$ and $F_{23}$, one may wish to construct, if they exist, 
trivariate \df's $F$ such that $F_{12}$ and $F_{23}$ are, respectively, the \df's of the first two and the last two 
components of the random triplet associated with $F$; the class of such functions $F$ is called \textit{Fr\'echet class} of
$F_{12}$ and $F_{23}$. An even harder problem is to construct such an $F$ when $F_{13}$ is also given, viz. when 
the \df\ of the first and the last component of that random triplet 
is also known. In such cases, and in all the cases when the marginals are \textit{overlapping}, the main problem is 
to determine a priori whether the given marginals are \textit{compatible}, viz. they can be derived from a 
common joint distribution. 

To the best of our knowledge, first results on the compatibility of three bivariate \df's
and on the corresponding Fr\'echet class were given by G. Dall'Aglio (1959) (compare also with (Dall'Aglio, 1972)),
and L. R\"uschendorf (1991a,1991b). In section 3 of the book by Joe (1997), the author studied in detail this case and some 
of its possible extensions to higher dimensions. 

In this paper, we aim at re-considering the foregoing problem in the class of \df's whose one-dimensional marginals 
are uniformly distributed on $[0,1]$: such \df's are called \textit{copulas}: see (Joe, 1997) and (Nelsen, 2006). 
This restriction does not cause any loss of generality in the problem because, thanks to \textit{Sklar's Theorem} 
(see (Sklar, 1959)), any multivariate \df\ can be represented by means of a copula and its one-dimensional marginals, and 
this representation is unique when the \df\ is continuous. Specifically, our goals are to:
\begin{itemize}
\item[(i)] determine under which conditions three  bivariate copulas $C_{12}$, $C_{13}$ and $C_{23}$ are \textit{compatible}, 
viz. they are the bivariate marginals of some trivariate copula $\widetilde C$;
\item[(ii)] construct the class of all trivariate copulas $\widetilde C$ with given bivariate marginals $C_{12}$, $C_{13}$ and $C_{23}$,
called the \textit{Fr\'echet class} of $C_{12}$, $C_{13}$ and $C_{23}$.
\end{itemize}
The main advantage of this approach completely based on copulas consists in the fact that 
it originates more intuitive and constructive procedures than in the previous literature (see the methods 
presented in section \ref{S:constr} and Theorem \ref{T:2fixcomp}), which permit easily to improve some 
known bounds (see Theorem \ref{T:bounds3cop}). 

We would like to stress that the above problems have a great interest in the development of copula theory, as 
underlined for example by Schweizer and Sklar (1983). Moreover, we also expect consequences 
in statistical applications, mainly when one wants to build a stochastic model from some knowledge about the kind of 
dependence exhibited by the involved random variables, and knows exactly certain marginal distributions. For example, 
constructions of \df's with given marginals are of relevance for the modelling of multivariate portfolio and bounding 
functions of dependent risks, such as the value at risk, the expected eccess of loss and other financial derivatives 
and risk measures (see (R\"uschendorf, 2004) and (McNeil et al., 2005)).

In Section \ref{S:first} we give some basic definitions, and then, we consider two constructions of copulas that 
will be useful in the sequel (Section \ref{S:constr}). In Section \ref{S:Comp} we present the characterization 
of the compatibility of three bivariate copulas, and we study the class of all trivariate copulas with given 
bivariate marginals (Section \ref{S:fre}). 

\section{Preliminaries}\label{S:first}

Let $n$ be in $\mathbb{N}$, $n\ge 2$, and denote by $\mathbf x=(x_1,\dots,x_n)$ any point in $\mathbb{R}^n$. An 
\textit{$n$--dimensional copula} (shortly, \textit{$n$--copula}) is a mapping $C_n:[0,1]^n\to [0,1]$ satisfying 
the following conditions:
\begin{itemize}
\item[(C1)] $C_n(\mathbf u)=0$ whenever $\mathbf u\in [0,1]^n$ has at least one component equal to $0$;
\item[(C2)] $C_n(\mathbf u)=u_i$ whenever $\mathbf u\in [0,1]^n$ has all components equal to $1$ except 
the $i$--th one, which is equal to $u_i$;
\item[(C3)] $C_n$ is \textit{$n$--increasing}, viz., for each $n$--box $B=\times_{i=1}^{n}[u_i,v_i]$ in $[0,1]^n$ 
with $u_i\le v_i$ for each $i\in \{1,\dots,n\}$,
\begin{equation}\label{E:n-in}
V_{C_n}\left( B\right):=\sum_{\mathbf z\in \times_{i=1}^{n}\{u_i,v_i\}} (-1)^{N(\mathbf z)} C_n(\mathbf z)\ge 0,
\end{equation}
where $N(\mathbf z)=card\{k\mid z_k=u_k\}$.
\end{itemize}

We denote by $\C_n$ the set of all $n$--dimensional copulas $(n\ge 2)$. For every $C_n\in \C_n$ and for 
every $\mathbf u\in [0,1]^n$, we have that

\begin{equation}\label{E:frechet}
W_n(\mathbf u)\le C_n(\mathbf u) \le M_n(\mathbf u),
\end{equation}
where
\begin{equation*}
W_n(\mathbf u):=\max\left\{\sum_{i=1}^{n}u_i-n+1,0\right\},\qquad
M_n(\mathbf u):=\min\{u_1,u_2,\dots,u_n\}.
\end{equation*}
Notice that $M_n$ is in $\C_n$, but $W_n$ is in $\C_n$ only for $n=2$. 
Another important $n$--copula is the product $\Pi_n(\mathbf u):=\prod_{i=1}^{n}u_i$. 

We recall that, for $C$ and $C'$ in $\C_2$, $C'$ is said to be greater than $C$ in the
\textit{concordance order}, and we write $C\preceq C'$, if $C(u_1,u_2)\le C'(u_1,u_2)$ for
all $(u_1,u_2)\in [0,1]^2$. Moreover, for $D$ and $D'$ in $\C_3$, $D'$ is said to be greater than $D$ in the
\textit{concordance order}, and we write $D\preceq D'$, if $D(\mathbf u)\le D'(\mathbf u)$ and 
$\overline{D}(\mathbf u)\le \overline{D'}(\mathbf u)$ for all $\mathbf u\in [0,1]^3$, 
where $\overline D$ is the survival copula of $D$ defined on $[0,1]^3$ by
\begin{eqnarray*}
\overline{D}(u_1,u_2,u_3)=&&1-u_1-u_2-u_3+D(u_1,u_2,1)+D(u_1,1,u_3)\\
&&+D(1,u_2,u_3)-D(u_1,u_2,u_3).
\end{eqnarray*}

For more details about copulas, see (Joe, 1997) and (Nelsen, 2006). 

Notice that, for each $C_n$ in $\C_n$, there exist a probability space $(\Omega,\mathcal A,P)$ and
a random vector $\mathbf U=(U_1,U_2,\dots,U_n)$, $U_i$ uniformly distributed on $[0,1]$ for 
every $i\in \{1,2,\dots,n\}$, such that $C_n$ is the \df\ of $\mathbf U$ (see (Billingsley, 1995)). 
As a consequence, for each $C_n\in \C_n$ and for each permutation $\mathbf{\sigma}=(\sigma_1,\dots,\sigma_n)$ of
$(1,2,\dots,n)$, the mapping $C_{n}^{\sigma}:[0,1]^n\to [0,1]$ given by
\[
C_{n}^{\sigma}(u_1,\dots,u_n)=C_n(u_{\sigma_1},\dots,u_{\sigma_n})
\]
is also in $C_n$. For example, if $C_3\in \C_3$, then the mapping $C_{3}^{(1,3,2)}$ given by
\[
C_{3}^{(1,3,2)}(u_1,u_2,u_3)=C_3(u_1,u_3,u_2)
\]
is also in $\C_3$. In particular, for each $C_2\in \C_2$, we write $C_{2}^{(2,1)}=C_{2}^{t}$, which is called
the \textit{transpose} of $C_2$. 

\begin{defn}
Let $C_{12}$, $C_{13}$ and $C_{23}$ be in $\C_{2}$. $C_{12}$, $C_{13}$ and $C_{23}$ are said to be \textit{compatible} 
if, and only if, there exists $\widetilde C\in \C_3$ such that, for all $u_1,u_2,u_3$ in $[0,1]$,
\begin{eqnarray}
C_{12}(u_1,u_2)&=&\widetilde C(u_1,u_2,1), \label{E:marg1}\\
C_{13}(u_1,u_3)&=&\widetilde C(u_1,1,u_3),\label{E:marg2}\\
C_{23}(u_2,u_3)&=&\widetilde C(1,u_2,u_3).\label{E:marg3}
\end{eqnarray} 
\end{defn}

In such a case, $C_{12}$, $C_{13}$ and $C_{23}$ are called the \textit{bivariate marginals} (briefly, $2$--marginals)
of $\widetilde C$.

Notice that $\Pi_2,\Pi_2,\Pi_2$ are compatible, because they are the $2$--marginals of $\Pi_3$. Analogously,
$M_2,M_2,M_2$ are compatible, because they are the $2$--marginals of $M_3$. The copulas $W_2,W_2,W_2$, however,
are not compatible (see (Schweizer and Sklar, 1983)).

\begin{defn}
Let $C_{12}$, $C_{13}$ and $C_{23}$ be in $\C_{2}$ such that they are compatible. The \textit{Fr\'echet class}
of $(C_{12}, C_{13},C_{23})$, denoted by $\F(C_{12}, C_{13},C_{23})$, is the class of all $\widetilde C\in \C_3$ 
such that \eqref{E:marg1}, \eqref{E:marg2} and \eqref{E:marg3} hold. 
\end{defn}

\section{Two constructions of copulas}\label{S:constr}

In this section, we introduce two constructions of copulas that shall be useful in the sequel.

\begin{prop}\label{P:prod}
Let $A$ and $B$ be in $\C_2$ and let $\mathbf C=\{C_t\}_{t\in [0,1]}$ be a family in $\C_2$. Then the mapping 
$A\ast_{\mathbf C} B:[0,1]^2\to [0,1]$ defined by 
\begin{equation}\label{E:Cprod}
(A\ast_{\mathbf C} B)(u_1,u_2)=\int_{0}^{1} C_t\left(\frac{\partial}{\partial t} A(u_1,t),
\frac{\partial}{\partial t} B(t,u_2)\right) \d t
\end{equation}
is in $\C_2$. 
\end{prop}

For a family $\mathbf C=\{C_t\}_{t\in [0,1]}$ in $\C_2$, $A\ast_{\mathbf C} B$ is called the 
\textit{$\mathbf C$--product} of the copulas $A$ and $B$. 
Given $C\in \C_2$, if $C_t=C$ for every $t$ in $[0,1]$, then we shall write 
$A\ast_{\mathbf C} B=A\ast_C B$. Notice that, if $C_t=\Pi_2$
for every $t$ in $[0,1]$, then the operation $\ast_{\Pi_2}$ is the product for copulas studied in (Darsow et al., 1992). 

\begin{prop}\label{P:Cprod}
Let $A$ and $B$ be in $\C_2$ and let $\mathbf C=\{C_t\}_{t\in [0,1]}$ be a family in $\C_2$. Then the mapping 
$A\star_{\mathbf C} B:[0,1]^3\to [0,1]$ defined by 
\begin{equation}\label{E:C3prod}
(A\star_{\mathbf C} B)(u_1,u_2,u_3)=\int_{0}^{u_2}C_t\left(\frac{\partial}{\partial t} A(u_1,t),\frac{\partial}{\partial t} B(t,u_3)\right) \d t
\end{equation}
is in $\C_3$.
\end{prop}
\begin{proof}
It is immediate that $A\star_{\mathbf C} B$ satisfies \textnormal{(C1)} and \textnormal{(C2)}. In order to prove \textnormal{(C3)} 
for $n=3$, let $u_i,v_i$ be in $[0,1]$ such that $u_i\le v_i$ for every $i\in \{1,2,3\}$. Since $A$ is $2$--increasing, we have that $A(v_1,t)-A(u_1,t)$ is increasing in $t\in [0,1]$, and, therefore, 
$\frac{\partial}{\partial t} A(v_1,t)\ge \frac{\partial}{\partial t} A(u_1,t)$ for all $t\in [0,1]$.
Analogously, $\frac{\partial}{\partial t} B(t,v_3)\ge \frac{\partial}{\partial t} B(t,u_3)$ for all $t\in [0,1]$.
Then, we have that
\begin{multline*}
V_{A\star_{\mathbf C} B}([u_1,v_1]\times[u_2,v_2]\times [u_3,v_3])\\
=\int_{u_2}^{v_2} V_{C_t}\left(\left[\frac{\partial}{\partial t} A(u_1,t),\frac{\partial}{\partial t} A(v_1,t)\right]
\times\left[\frac{\partial}{\partial t} B(t,u_3),\frac{\partial}{\partial t} B(t,v_3)\right]\right)\d t\ge 0,
\end{multline*}
which concludes the proof. 
\end{proof}

For a family $\mathbf C=\{C_t\}_{t\in [0,1]}$ in $\C_2$, $A\star_{\mathbf C} B$ is called the 
\textit{$\mathbf C$--lifting} of the copulas $A$ and $B$. 
Given $C\in \C_2$, if $C_t=C$ for every $t$ in $[0,1]$, we shall write 
$A\star_{\mathbf C} B=A\star_C B$. Notice that, if $C_t=\Pi_2$ for every $t$ in $[0,1]$, 
then the operation $\star_{\Pi_2}$ was considered in (Darsow et al., 1992) and (K\'olesarov\'a et al., 2006). 
Notice that the copula given by \eqref{E:C3prod} has an interpretation in terms of mixtures 
of conditional distributions (see section 4.5 of (Joe, 1997)). 
Moreover, we easily derive the following result, which, as a byproduct, also proves Proposition \ref{P:prod}.

\begin{prop}
Let $A$ and $B$ be in $\C_2$ and let $\mathbf C=\{C_t\}_{t\in [0,1]}$ be a family in $\C_2$. Then the $2$--marginals of
$A\star_{\mathbf C} B$ (that are $2$--copulas) are $A$, $A\ast_{\mathbf C} B$ and $B$, respectively.
\end{prop}

Finally, we show a result that will be useful in next section, 
concerning the concordance order between two $3$--copulas generated by means of the
$\mathbf C$--lifting operation.

\begin{prop}\label{P:boundstar}
Let $\mathbf C=\{C_t\}_{t\in [0,1]}$ and $\mathbf C'=\{C_{t}'\}_{t\in [0,1]}$ be two families in $\C_2$. 
If $C_t\preceq C_{t}'$ for every $t$ in $[0,1]$, then, for all $A$ and $B$ in $\C_2$, 
$A\star_{\mathbf C} B\preceq A\star_{\mathbf C'} B$.
\end{prop}
\begin{proof}
It is immediate that $C_t\preceq C_{t}'$, for every $t\in [0,1]$, implies 
$A\star_{\mathbf C} B\le A\star_{\mathbf C'} B$ in the pointwise order.
Thus, we have only to prove that $\overline{A\star_{\mathbf C} B}\le \overline{A\star_{\mathbf C'} B}$. 
To this end, notice that 
\begin{eqnarray*}
&& (A\ast_{\mathbf C} B)(u_1,u_2,1)=(A\ast_{\mathbf C'} B)(u_1,u_2,1)=A(u_1,u_2),\\
&& (A\ast_{\mathbf C} B)(1,u_2,u_3)=(A\ast_{\mathbf C'} B)(1,u_2,u_3)=B(u_2,u_3).
\end{eqnarray*}
Therefore $\overline{A\star_{\mathbf C} B}(u_1,u_2,u_3)\le \overline{A\star_{\mathbf C'} B}(u_1,u_2,u_3)$ 
if, and only if,
\[
(A\ast_{\mathbf C} B)(u_1,u_3)-(A\star_{\mathbf C} B)(u_1,u_2,u_3)\le (A\ast_{\mathbf C'} B)(u_1,u_3)-
(A\star_{\mathbf C'} B)(u_1,u_2,u_3),
\]
which, in turn, is equivalent to
\[
\int_{u_2}^{1}C_t\left(\frac{\partial}{\partial t} A(u_1,t),\frac{\partial}{\partial t} B(t,u_3)\right) \d t\le 
\int_{u_2}^{1}C_{t}'\left(\frac{\partial}{\partial t} A(u_1,t),\frac{\partial}{\partial t} B(t,u_3)\right) \d t,
\]
and this is obviously true since $C_t\preceq C_{t}'$ for every $t\in [0,1]$.
\end{proof}

Notice that the latter results are interesting in their own right. Specifically, they allow us to construct
families of bivariate and trivariate copulas starting with known bivariate copulas (see (Durante et al., 2007) for details).

In the case of distribution functions with given densities, similar constructions were originally proposed by Joe (1996), 
and later developed in detail by Bedford and Cooke (2001, 2002), Aas et al. (2007) and Berg and Aas (2007). These constructions, 
which are formulated in the multivariate case, are based on a decomposition of a multivariate $d$--dimensional 
density $(d\ge 3)$ into $\tfrac{d(d - 1)}{2}$ bivariate copula densities.

\section{Compatibility of bivariate copulas}\label{S:Comp}

In order to determine conditions under which three $2$--copulas are compatible,
we start by characterizing the class $\C(C_{12},C_{23})$ of all $2$--copulas $C_{13}$ that are compatible with $C_{12}$ and $C_{23}$.

\begin{thm}\label{T:2fixcomp} 
Let $C_{12}$ and $C_{23}$ be in $\C_2$. A $2$--copula $C_{13}$ is in $\C(C_{12},C_{23})$ if, and only if, 
there exists a family $\mathbf C=\{C_t\}_{t\in [0,1]}$ in $\C_2$ such that 
\begin{equation}\label{E:bivmarg}
C_{13}=C_{12}\ast_{\mathbf C} C_{23}. 
\end{equation}
\end{thm}
\begin{proof}
If $C_{12}$, $C_{13}$ and $C_{23}$ are compatible, then there exists $\widetilde C\in \C_3$ 
such that \eqref{E:marg1}, \eqref{E:marg2} and \eqref{E:marg3} hold. Then there exist a probability space $(\Omega,\mathcal F,P)$ 
and a random vector $\mathbf U=(U_1,U_2,U_3)$, $U_i$ uniformly distributed on $[0,1]$ 
for each $i\in \{1,2,3\}$, such that, for all $u_1,u_2,u_3$ in $[0,1]$,
\begin{equation}\label{E:T1}
\widetilde C(u_1,u_2,u_3)=P(U_1\le u_1,U_2\le u_2,U_3\le u_3),
\end{equation}
and $C_{12}$ is the copula of $(U_1,U_2)$, $C_{13}$ is the copula of $(U_1,U_3)$ and $C_{23}$ is the copula of $(U_2,U_3)$. 
Then we have that
\begin{equation}\label{E:T2}
\widetilde C(u_1,u_2,u_3)=\int_{0}^{u_2} C_{t}(P(U_1\le u_1 \mid U_2=t),P(U_3\le u_3\mid U_2=t))\,\d t,
\end{equation}
where, for each $t\in [0,1]$, $C_{t}$ is the copula associated with the (conditional) distribution function of $(U_1,U_3)$ given $U_2=t$.
But, by simple calculations, we also obtain that, almost surely on $[0,1]$,
\[
P(U_1\le u_1\mid U_2=t)=\frac{\partial C_{12}(u_1,t)}{\partial t},\quad P(U_3\le u_3\mid U_2=t)=\frac{\partial C_{23}(t,u_3)}{\partial t}.
\]
Therefore we can rewrite \eqref{E:T2} in the form
\[
\widetilde C(u_1,u_2,u_3)=\int_{0}^{u_2}C_{t} \left( \frac{\partial }{\partial t}C_{12}(u_1,t),\frac{\partial }{\partial t}C_{23}(t,u_3)\right)\,\d t,
\]
and, hence, we obtain
\[
C_{13}(u_1,u_3)=\int_{0}^{1}C_{t} \left( \frac{\partial }{\partial t}C_{12}(u_1,t),\frac{\partial }{\partial t}C_{23}(t,u_3)\right)\,\d t,
\]
and therefore Eq. \eqref{E:bivmarg} holds.

In the other direction, suppose that there exists $\mathbf C=\{C_{t}\}_{t\in [0,1]}$ in $\C_2$ such that 
$C_{13}=C_{12}\ast_{\mathbf C} C_{23}$. From Proposition \ref{P:Cprod}, the function $\widetilde C$ given by
\[
\widetilde C(u_1,u_2,u_3)=(C_{12}\star_{\mathbf C} C_{23})(u_1,u_2,u_3)
\]
is a $3$--copula whose $2$--marginals are, respectively, $C_{12}$, $C_{13}$ and $C_{23}$, showing that
they are compatible.
\end{proof}

Note that the family $\mathbf C=\{C_t\}_{t\in [0,1]}$ is not completely arbitrary and depends, of course, 
on the copulas $C_{12}$ and $C_{23}$. 

\begin{cor}\label{C:2fixbounds}
For each $C_{13}$ in $\C(C_{12},C_{23})$ we have that
\begin{equation}\label{E:PAB-bounds}
C_{12}\ast_{W_2} C_{23}\preceq C_{13}\preceq C_{12} \ast_{M_2} C_{23},
\end{equation}
and these bounds are sharp.
\end{cor}

Therefore, we obtain the following characterization.

\begin{thm}\label{T:char}
Let $C_{12}$, $C_{13}$ and $C_{23}$ be in $\C_{2}$. $C_{12}$, $C_{13}$ and $C_{23}$ are compatible if, and only if,
there exist three families of $2$--copulas,
\[
\mathbf C_1=\{C^{1}_{t}\}_{t\in[0,1]},\quad \mathbf C_2=\{C^{2}_{t}\}_{t\in[0,1]},\quad \mathbf C_3=\{C^{3}_{t}\}_{t\in[0,1]},
\]
such that
\begin{equation}\label{E:comp}
C_{12}=C_{13}\ast_{\mathbf C_3} C_{32},\quad C_{13}=C_{12}\ast_{\mathbf C_2} C_{23},\quad C_{23}=C_{21}\ast_{\mathbf C_1} C_{13},
\end{equation}
where, for $1\le i<j\le 3$, $C_{ji}:=C_{ij}^t$.
\end{thm}
\begin{proof}
Notice that $C_{12}$, $C_{13}$ and $C_{23}$ are compatible if, and only if, $C_{12}\in \C(C_{13},C_{23})$, 
$C_{13}\in \C(C_{12},C_{23})$ and $C_{23}\in \C(C_{12},C_{13})$. Now, the assertion can be proved by means of 
a slight modification of the proof of Theorem \ref{T:2fixcomp}.
\end{proof}

In general, it is a difficult task to check whether a copula $C_{13}$ is compatible with two copulas $C_{12}$ and $C_{23}$.
However, Corollary \ref{C:2fixbounds} gives us some information: in fact, in order to prove that
$C_{13}\notin \C(C_{12},C_{23})$, it suffices to find a point $(u,v)$ in $[0,1]^2$ such that $C_{13}(u,v)$ violates 
\eqref{E:PAB-bounds}. 

\begin{exmp}
Let $C_{12}$ be the copula given by
\[
C_{12}(u_1,u_2)=u_1 u_2+u_1 u_2(1-u_1)(1-u_2),
\]
let $C_{23}$ be equal to the product copula $\Pi_2$, and let $C_{13}^{\alpha}$ be the Clayton copula given by
\[
C_{13}^{\alpha}(u_1,u_3)=(u_{1}^{-\alpha}+u_{3}^{-\alpha}-1)^{-1/\alpha}
\] 
for $\alpha\ge 0$. For a sufficiently large $\alpha$, the above three copulas are not compatible.
In fact, following Corollary \ref{C:2fixbounds}, we have that
\[
(C_{12}\ast_{M_2}C_{23})\left(\tfrac{1}{2},\tfrac{1}{2}\right)=\tfrac{7}{16},
\]
while $C_{13}^{\alpha}$ tends to $\tfrac{1}{2}$ when $\alpha$ tends to $+\infty$.
\end{exmp}

\smallskip
\begin{rem}
Theorem \ref{T:2fixcomp} was originally formulated by Dall'Aglio (1959), where also
Corollary \ref{C:2fixbounds} was presented (for the latter, see also (R\"uschendorf, 1991a)).
\end{rem}

\section{Fr\'echet class of three bivariate copulas}\label{S:fre}

Given three compatible $2$--copulas $C_{12}$, $C_{13}$ and $C_{23}$, we are now interested on 
the \textit{Fr\'echet class} $\F(C_{12},C_{13},C_{23})$ of all $3$--copulas whose $2$--marginals are, 
respectively, $C_{12}$, $C_{13}$ and $C_{23}$. As before, we first consider the class $\F(C_{12},C_{23})$
of all trivariate copulas whose $2$--marginals $C_{12}$ and $C_{23}$ are known.

\begin{thm}\label{T:2fix3cop} 
Let $C_{12}$ and $C_{23}$ be in $\C_2$. A $3$--copula $\widetilde C$ is in $\F(C_{12},C_{23})$ 
if, and only if, there exists a family $\mathbf C=\{C_t\}_{t\in [0,1]}$ in $\C_2$ such that 
\begin{equation}\label{E:32fix}
\widetilde C=(C_{12}\star_{\mathbf C} C_{23}).
\end{equation}
Moreover, for every $\widetilde C$ in $\F(C_{12},C_{23})$ and for all $u_1,u_2$ and $u_3$ in $[0,1]$, 
\begin{equation}\label{E:2fix3bounds}
(C_{12}\star_{W_2} C_{23})(u_1,u_2,u_3)\le \widetilde C(\mathbf u)\le (C_{12}\star_{M_2} C_{23})(u_1,u_2,u_3)
\end{equation}
and the bounds are sharp.
\end{thm}

The above theorem is simply obtained by reconsidering the proof of Theorem \ref{T:2fixcomp}. Notice that the bounds
\eqref{E:2fix3bounds} have also been obtained in (R\"uschendorf, 1991a)(Proposition 7) (see also (Joe, 1997)(Theorem 3.10)). 

Theorem \ref{T:2fix3cop} gives a powerful constructive way to determine all $3$--copulas with two given bivariate marginals. 
For example, if $C_{12}=C_{23}$, the copulas given by \eqref{E:32fix} are all possible trivariate copulas that can be used 
in the construction of Markov chains of second order (see section 8.1 in (Joe, 1997)).

Moreover, we can also easily derive that, if either $C_{12}$ or $C_{23}$ are shuffles of Min, then $\mathcal F(C_{12},C_{23})$
is formed just by one element (compare with (Durante et al., 2007), (Koles\'arov\'a et al., 2006)).

As a consequence of Theorem \ref{T:2fix3cop}, we can also state the following result.

\begin{thm}\label{T:char3cop}
Let $C_{12}$, $C_{13}$ and $C_{23}$ be three compatible $2$--copulas. 
A $3$--copula $\widetilde C$ is in $\F(C_{12},C_{13},C_{23})$ if, and only if, there exist three families of $2$--copulas,
\[
\mathbf C_1=\{C^{1}_{t}\}_{t\in[0,1]},\quad \mathbf C_2=\{C^{2}_{t}\}_{t\in[0,1]},\quad \mathbf C_3=\{C^{3}_{t}\}_{t\in[0,1]},
\]
such that
\begin{equation}\label{E:comp3cop}
\widetilde C= (C_{13}\star_{\mathbf C_3}C_{32})^{(1,3,2)}= C_{12}\star_{\mathbf C_2}C_{23}=(C_{21}\star_{\mathbf C_1}C_{13})^{(2,1,3)}.
\end{equation}
\end{thm}

Now, we give pointwise lower and upper bounds for $\F(C_{12},C_{13},C_{23})$. 

\begin{thm}\label{T:bounds3cop}
For every $\widetilde C\in \F(C_{12},C_{13},C_{23})$ and for all $u_1,u_2,u_3$ in $[0,1]$, we have
\begin{equation}\label{E:bounds}
C_L(u_1,u_2,u_3)\le \widetilde C(u_1,u_2,u_3)\le C_U(u_1,u_2,u_3),
\end{equation}
where
\begin{eqnarray*}
C_L(u_1,u_2,u_3)= &&\max_{(i,j,k)\in \mathcal P}\{(C_{ij}\star_{W_2}C_{jk})(u_i,u_j,u_k),(C_{ij}\star_{M_2}C_{jk})(u_i,u_j,u_k)\\
&&+C_{ik}(u_i,u_k)-(C_{ij}\ast_{M_2}C_{jk})(u_i,u_k)\},
\end{eqnarray*}
\begin{eqnarray*}
C_U(u_1,u_2,u_3)= &&\min_{(i,j,k)\in \mathcal P}\{(C_{ij}\star_{M_2}C_{jk})(u_i,u_j,u_k),(C_{ij}\star_{W_2}C_{jk})(u_i,u_j,u_k)\\
&&+C_{ik}(u_i,u_k)-(C_{ij}\ast_{W_2}C_{jk})(u_i,u_k)\},
\end{eqnarray*}
and $\mathcal P=\{(1,2,3),(1,3,2),(2,1,3)\}$.
\end{thm}
\begin{proof}
If $\widetilde C\in \F(C_{12},C_{13},C_{23})$, then, from Theorem \ref{T:char3cop}, 
there exist three families of $2$--copulas, such that $\widetilde C$ can be
expressed in the forms \eqref{E:comp3cop}.

Since $W_2\preceq C \preceq M_2$ for every $C\in \C_2$, Proposition \ref{P:boundstar}
ensures that, for each $(i,j,k)$ in $\mathcal P$,
\[
(C_{ij}\star_{W_2}C_{jk})^{(i,j,k)}\preceq \widetilde C\preceq (C_{ij}\star_{M_2}C_{jk})^{(i,j,k)}.
\]
Therefore, for each $(i,j,k)$ in $\mathcal P$ and $\mathbf u\in [0,1]^3$, we have that
\begin{equation}\label{E:bb1}
(C_{ij}\star_{W_2}C_{jk})(u_i,u_j,u_k)\le \widetilde C(\mathbf u)\le (C_{ij}\star_{M_2}C_{jk})(u_i,u_j,u_k).
\end{equation}
and
\begin{equation}\label{E:bb2}
\overline{(C_{ij}\star_{W_2}C_{jk})}(u_i,u_j,u_k)\le \overline{\widetilde C}(\mathbf u)\le 
\overline{(C_{ij}\star_{M_2}C_{jk})}(u_i,u_j,u_k).
\end{equation}

The left hand side of \eqref{E:bb2} is equivalent to:
\begin{eqnarray*}
&& 1-u_1-u_2-u_3+C_{ij}(u_i,u_j)+C_{jk}(u_j,u_k)+(C_{ij}\ast_{W_2}C_{jk})(u_i,u_k)\\
&& -(C_{ij}\star_{W_2}C_{jk})(u_i,u_j,u_k)\\
&\le& 1-u_1-u_2-u_3+C_{ij}(u_i,u_j)+C_{jk}(u_j,u_k)+C_{ik}(u_i,u_k)-\widetilde C(u_i,u_j,u_k).
\end{eqnarray*}

The right hand side of \eqref{E:bb2} is equivalent to:
\begin{eqnarray*}
&& 1-u_1-u_2-u_3+C_{ij}(u_i,u_j)+C_{jk}(u_j,u_k)+C_{ik}(u_i,u_k)-\widetilde C(u_i,u_j,u_k)\\
&\le& 1-u_1-u_2-u_3+C_{ij}(u_i,u_j)+C_{jk}(u_j,u_k)+(C_{ij}\ast_{M_2}C_{jk})(u_i,u_k)\\
&& -(C_{ij}\star_{M_2}C_{jk})(u_i,u_j,u_k).
\end{eqnarray*}

Easy calculations show that these inequalities are equivalent to:
\begin{eqnarray*}\label{E:bb3}
&&\widetilde C(\mathbf u)\le (C_{ij}\star_{W_2}C_{jk})(u_i,u_j,u_k)+C_{ik}(u_i,u_k)-(C_{ij}\ast_{W_2}C_{jk})(u_i,u_k),\\
&&\widetilde C(\mathbf u)\ge (C_{ij}\star_{M_2}C_{jk})(u_i,u_j,u_k)+C_{ik}(u_i,u_k)-(C_{ij}\ast_{M_2}C_{jk})(u_i,u_k)\label{E:bb4}.
\end{eqnarray*}
Using these inequalities and \eqref{E:bb1}, we directly get \eqref{E:bounds}.
\end{proof}

In Theorem 3.11 in (Joe, 1997), the author provided an upper bound $F_U$ and a lower bound $F_L$ for $\F(C_{12},C_{13},C_{23})$
given by
\begin{eqnarray}
&& F_U(u_1,u_2,u_3)= \min\{C_{12}(u_1,u_2),C_{13}(u_1,u_3),C_{23}(u_2,u_3), 1-u_1\nonumber\\
&& -u_2-u_3+C_{12}(u_1,u_2)+C_{13}(u_1,u_3)+C_{23}(u_2,u_3)\}\label{E:Joe1}\\
&& F_L(u_1,u_2,u_3)= \max\{0,C_{12}(u_1,u_2)+C_{13}(u_1,u_3)-u_1,C_{12}(u_1,u_2)\nonumber\\
&&+C_{23}(u_2,u_3)-u_2, C_{13}(u_1,u_3)+C_{23}(u_2,u_3)-u_3\}\label{E:Joe2}.
\end{eqnarray}
In the following result, we show that the bounds \eqref{E:bounds} improve the bounds given by \eqref{E:Joe1} and \eqref{E:Joe2}.

\begin{prop}
Let $C_{12}$, $C_{13}$ and $C_{23}$ be three compatible $2$--copulas. For every $\mathbf u\in [0,1]^3$, 
we have that $C_L(\mathbf u)\ge F_L(\mathbf u)$ and $C_U(\mathbf u)\le F_U(\mathbf u)$.
\end{prop}
\begin{proof}
Let $\mathbf u$ be in $[0,1]^3$. We have that
\begin{eqnarray*}
C_L(\mathbf u)&\ge& (C_{13}\star_{W_2}C_{32})(u_1,u_3,u_2)\\
&=& \int_{0}^{u_3}W_2\left(\frac{\partial}{\partial t} 
C_{13}(u_1,t),\frac{\partial}{\partial t} C_{32}(t,u_2)\right) \d t\\
&\ge&  C_{13}(u_1,u_3)+C_{23}(u_2,u_3)-u_3,
\end{eqnarray*}
and, analogously,
\begin{eqnarray*}
C_L(\mathbf u)&\ge& C_{12}(u_1,u_2)+C_{13}(u_1,u_3)-u_1,\\
C_L(\mathbf u)&\ge& C_{12}(u_1,u_2)+C_{23}(u_2,u_3)-u_2.
\end{eqnarray*}
Therefore, since $C_L(\mathbf u)\ge 0$, it follows that $C_L(\mathbf u)\ge F_L(\mathbf u)$
for every $\mathbf u$ in $[0,1]^3$.

On the other hand, we have that
\begin{eqnarray*}
C_U(\mathbf u)&\le& (C_{13}\star_{M_2}C_{32})(u_1,u_3,u_2)\\
&=& \int_{0}^{u_3}\min\left(\frac{\partial}{\partial t} C_{13}(u_1,t),\frac{\partial}{\partial t} C_{32}(t,u_2)\right) \d t\\
&\le& \min(C_{13}(u_1,u_3),C_{23}(u_2,u_3)),
\end{eqnarray*}
and, analogously, $C_U(\mathbf u)\le C_{12}(u_1,u_2)$. Moreover, for every $\mathbf u\in [0,1]^3$, we have that
\begin{eqnarray*}
&& (C_{12}\star_{W_2}C_{23})(u_1,u_2,u_3)+C_{13}(u_1,u_3)-(C_{12}\ast_{W_2}C_{23})(u_1,u_3)\\
&&\le 1-u_1-u_2-u_3+C_{12}(u_1,u_2)+C_{13}(u_1,u_3)+C_{23}(u_2,u_3),
\end{eqnarray*}
as a consequence of the fact that $\overline{(C_{12}\star_{W_2}C_{23})}(\mathbf u)\ge 0$. Thus $C_U(\mathbf u)\le F_U(\mathbf u)$ 
for every $\mathbf u$ in $[0,1]^3$.
\end{proof}

\begin{exmp}
From Theorem \ref{T:bounds3cop}, if $\widetilde C$ is in $\F(\Pi_2,\Pi_2,\Pi_2)$, 
then, for every $u_1$, $u_2$ and $u_3$ in $[0,1]$, we have
\[
C_L(u_1,u_2,u_3)\le \widetilde C(u_1,u_2,u_3)\le C_U(u_1,u_2,u_3),
\]
where 
\begin{eqnarray}
C_L(u_1,u_2,u_3)&=& \max\{u_1 W_2(u_2,u_3), u_2 W_2(u_1,u_3), u_3 W_2(u_1,u_2)\},\\
C_U(u_1,u_2,u_3)&=& \min\{u_1 M_2(u_2,u_3), u_2 M_2(u_1,u_3), u_3 M_2(u_1,u_2)\}.
\end{eqnarray}
It is easy to check that, in this case, $C_L=F_L$ and $C_U=F_U$. These bounds were also obtained 
in (Rodr\'iguez-Lallena and \'Ubeda-Flores, 2004), by making different calculations (compare also 
with section 3.4.1 in (Joe, 1997)).
In particular, it was stressed in (Rodr\'iguez-Lallena and \'Ubeda-Flores, 2004) that $C_L$ and $C_U$ may not be copulas. 
\end{exmp} 

\section*{Acknowledgements}
The authors are grateful to Prof. C. Genest and Prof. R.B. Nelsen for their comments on a first
version of this manuscript. Moreover, the first author kindly acknowledges Prof. L. R\"uschendorf 
for fruitful discussions and for drawing our attention to previous results in this context. 
The third author acknowledges the support by the Ministerio de Educaci\'on y Ciencia (Spain) 
and FEDER, under research project MTM2006-12218.

\small{
}

\end{document}